\documentclass{article}
\usepackage{euscript,amsmath,amssymb,amsfonts,graphicx,bm,mathtools, amsthm}
  \usepackage{hyperref}
\usepackage{latexsym,amssymb,enumerate,amsmath} 
  \usepackage{graphicx} 
  \usepackage{tikz}
  \usepackage[square,sort,comma,numbers]{natbib}
\usepackage{pgfplots}
\usepackage{fancyvrb}

\usepackage{bbm}

\numberwithin{equation}{section}

\newtheorem{proposition}{Proposition}
\newtheorem{theorem}{Theorem}
\newtheorem{corollary}{Corollary}
\newtheorem{lemma}{Lemma}

\newcommand{\dd}{\textup{d}}

\def\E{\mathbb{E}}

\def\R{\mathbb{R}}

\def\<{\langle}
\def\>{\rangle}

\def\kmin{\textup{kmin}}
\def\binom{\textup{Binomial}}
\def\geo{\textup{Geometric}}
\def\Var{\textup{Var}}

\usepackage{amsfonts}
\usepackage{graphicx}
\usepackage{epstopdf}
\usepackage{algorithmic}

\begin{document}

\title{Passage times of fast inhomogeneous immigration processes}

\author{
Hwai-Ray Tung\thanks{Department of Mathematics, University of Utah, Salt Lake City, UT 84112 USA (\texttt{ray.tung@utah.edu}).}\and
Sean D. Lawley\thanks{Department of Mathematics, University of Utah, Salt Lake City, UT 84112 USA (\texttt{lawley@math.utah.edu}).}
}
\date{\today}
\maketitle

\begin{abstract}
In many biophysical systems, key events are triggered when the fastest of many random searchers find a target. Most mathematical models of such systems assume that all searchers are initially present in the search domain, which permits the use of classical extreme value theory. In this paper, we explore $k$th passage times of inhomogeneous immigration processes where searchers are added to the domain over time either through time inhomogeneous rates or a Yule (pure birth) process. We rigorously prove convergence in distribution and convergence of moments of the $k$th passage times for both processes as immigration rates grow. In particular, we relate immigration with time inhomogeneous rates to previous work where all searchers are initially present through a coupling argument and demonstrate how immigration through a Yule process can be viewed as a time inhomogeneous immigration process with a random time shift. For Yule immigration, we find that the extreme distributions depart from the classical family of Frechet, Gumbel, and Weibull, and we compare our results to classical theorems on branching Brownian motion. This work offers one of the few examples where extreme value distributions can be obtained exactly for random variables which are neither independent nor identically distributed. 
\end{abstract}



\newpage
\section{Introduction}

The timing of events can often be cast as the first time a random ``searcher'' finds a ``target,'' which is called a first passage time (FPT) \cite{redner2001}. To name only a few, applications as varied as animal foraging \cite{kurella2015, mckenzie2009}, asset trading \cite{li2019first}, and menopause timing \cite{lawley2023slow, johnson2025modeling} have all been studied using FPT theory. Importantly, the search ``space'' may not necessarily be physical space, such as in models of carcinogenesis \cite{armitage1954age, nordling1953new, durrett2015branching} and computer search algorithms \cite{kao1996}. Traditionally, mathematical analysis of FPTs aims to determine the statistics of a single searcher. For instance, the mean FPT quantifies how long it takes a typical searcher to find a target, thereby neglecting outliers which find the target unusually quickly or slowly.

However, the important timescale in many systems is the time it takes the fastest searcher out of many searchers to find a target, which is called an extreme FPT \cite{weiss1983, lawley2020dist}. Extreme FPTs arise in a variety of biophysical systems in which the ``winner takes all,'' including gene regulation \cite{harbison2004}, human fertilization \cite{meerson2015}, and decision-making \cite{karamched2020heterogeneity, linn2024fast}. See \cite{grebenkov2024target} for additional examples.

Mathematical models of extreme FPTs commonly assume that there are $N\gg1$ noninteracting searchers initially present in the domain. This formalism is convenient because it makes the fastest FPT equivalent to the minimum of $N$ independent and identically distributed (iid) FPTs, which allows the use of classical extreme value theory \cite{colesbook}. While suitable for some applications, assuming that every searcher is initially present (i.e.\ that each searcher begins searching at the same time) is not appropriate for systems in which searchers are introduced progressively over time. As recently noted in Ref.~\cite{grebenkov2025fastest}, gene regulation, viral infection, and cell signaling all involve many searchers which begin searching at different times.

In this paper, we study extreme FPTs where the searchers progressively immigrate into the system. We study two types of immigration, which we call time-inhomogeneous immigration (TII) and Yule immigration (YI). In TII, searchers enter the system according to a Poisson point process with time-dependent rate $\lambda u(t)$. This means that the probability that a searcher enters the system between time $t$ and time $t+\Delta t$ is approximately $\lambda u(t)\Delta t$ for small $\Delta t$. TII generalizes the immigration model introduced by Campos and Mendez \cite{campos2024} and further explored in \cite{tung2025first}, which assumed an immigration rate that is constant in time (i.e.\ $u(t)=1$ for all time $t$).

In YI, searchers enter the system at rate $\lambda N(t)$, where $N(t)$ is the random number of searchers in the system at time $t$. That is, the immigration rate is proportional to the population size, and the population grows according to a so-called Yule process \cite{ross2009introduction}. YI is thus natural for systems in which the number of searchers grows by births, replications, or infections of a contagion.

For both TII and YI, we determine the distribution and all the moments of the $k$th fastest FPTs for fast immigration (i.e.\ in the large $\lambda$ limit). The proofs for the TII results revolve around establishing a coupling with previous models featuring $N$ initial searchers and no immigration, while the proofs for YI results are analytic in nature. 
Importantly, classical extreme value theory is not directly applicable to either TII or YI, since the search times are neither independent nor identically distributed. 
Indeed, our results constitute a rare instance in which extreme value statistics can be computed exactly for strongly correlated random variables \cite{majumdar2020extreme}. From the perspective of extreme value theory, it is also noteworthy that the limiting search time distribution for YI is not one of the three ``universal'' distributions \cite{colesbook} (i.e.\ the fastest FPT turns out to obey neither a Gumbel, Frechet, or Weibull distribution).

The rest of the paper is organized as follows. We formulate the model in section~\ref{sec:model}. We state our rigorous results for TII and YI in sections~\ref{sec:TIresults} and \ref{sec:YIresults}, respectively. In section \ref{sec:YIresults}, we also compare our results on YI to results on branching processes, including classical results on branching Brownian motion \cite{bramson1978maximal, bramson1983convergence, lalley1987conditional, berestycki2017branching, kim2023maximum}. We compare our theory to numerical solutions in section~\ref{sec:numerics} for some canonical models of stochastic search. We conclude by discussing relations to prior work in section~\ref{sec:discussion}. The appendices present the proofs of the results in sections \ref{sec:TIresults} and \ref{sec:YIresults}.

\section{Model}
\label{sec:model} 
In our model, searchers immigrate into the system over time and start at an initial position sampled from a given probability distribution (which could be a Dirac delta mass at a single point). These searchers randomly and independently explore the spatial domain under the same probability laws until they reach some given target region(s). We are interested in $T_k$, the time it takes the fastest $k\ge1$ searchers to find the target, when searchers immigrate quickly. 

More precisely, let $\sigma_n$ be the time that the $n$th particle immigrates in, and let $\tau_n$ be the time it takes the $n$th particle to find the target after immigrating in. The collection of search times is thus given by
\begin{align}\label{eq:collection}
    \{\tau_1 + \sigma_1,\tau_2+\sigma_2, \tau_3+\sigma_3, \dots\}
    =\{\tau_n+\sigma_n\}_{n\ge1},
\end{align}
where $\{\tau_n\}_{n\ge1}$ are iid realizations of $\tau$ and $\{\sigma_n\}_{n\ge1}$ is a sequence of immigration times whose distributions which depend on the immigration scheme. Notice that the search times $\tau_n+\sigma_n$ are neither independent nor identically distributed. The fastest search time is then
\begin{align}\label{eq:T1}
    T_1
    =\min\{\tau_n+\sigma_n\},
\end{align}
and the $k$-th fastest search time is
\begin{align*}
    T_k
    =\min\big\{\{\tau_n+\sigma_n\big\}\big\backslash\cup_{j=1}^{k-1}T_j\},\quad k\ge1.
\end{align*}

As such, knowing $\tau$ and $\{\sigma_n\}_{n\ge1}$ fully describes our system. To understand $\tau$, we examine its survival probability,
\begin{align*}
    S(t)
    =P(\tau>t),
\end{align*}
which encompasses information like the initial searcher position distribution, spatial domain, searcher motion, and target location. The assumptions on $S(t)$ and scenarios where they apply are discussed in Section \ref{sec:indivSearch}. For $\{\sigma_n\}_{n\ge1}$, we examine two immigration schemes - immigration rates that change proportional to a specified function of time $u(t)$, or immigration rates that change proportional to the population size. These are described in Sections \ref{sec:TIimmigrationrates} and \ref{sec:YIimmigrationrates}.

\subsection{Individual searcher survival $S(t)$}
\label{sec:indivSearch}
Let $f\sim g$ indicate that the functions $f$ and $g$ have the same small time behavior, i.e.\ $\lim_{t\rightarrow 0^+}f/g =1$. As in \cite{tung2025first}, we examine two cases for $S(t)$,
\begin{align}
    1-S(t) &\sim At^p, &&\quad A>0,\, p>0, \label{eq:power}\\
    1-S(t) &\sim At^pe^{-C/t}, &&\quad A>0,\, C>0,\,p\in\R. \label{eq:exponential}
\end{align}
We only need the small time behavior for $1-S(t)$ since we are interested in the case where immigration is fast, and therefore the target is always found quickly. This makes searcher behavior at larger times irrelevant. 

Three examples of processes with the power law decay in \eqref{eq:power} are searches on a discrete graph with exponential jump rates \cite{lawley2020networks}, L{\'e}vy flights \cite{lawley2023super}, and diffusive searches where searchers can start arbitrarily close to the target \cite{madrid2020comp}. For discrete graph searches, $p\ge1$ is the shortest number of jumps from the initial position of the searcher to the target, and $A$ is found by noting $Ap!$ is the sum of the product of the rates of a path over all paths of length $p$ that run from the initial position to the target \cite{lawley2020networks}. For L{\'e}vy flights, $p=1$ and an expression is known for $A$ \cite{lawley2023super}. L{\'e}vy flights have been used to model animal foraging \cite{palyulin2014, palyulin2016, tzou2023, gomez2024first, metzler2004} since their paths tend to linger in a small area before making a large movement to a new area. For the last example with diffusive searchers, this situation arises if the initial position of the search is uniformly distributed on the domain \cite{weiss1983, madrid2020comp, grebenkov2020single}. 

The exponential decay in \eqref{eq:exponential} covers cases of diffusive searchers where searchers cannot start arbitrarily close to the target \cite{lawley2020uni}. Studying the FPTs of these cases have been of longstanding interest \cite{redner2001, bressloffbook}. Factors like obstacles, drift, and dimension may affect the values of $A, p,$ and $C$, but do not change that $1-S(t) \sim At^pe^{-C/t}$, making this a very general category. In the case of $d$-dimensional Brownian motion, $C>0$ is the diffusion timescale,
\begin{equation}
    C=\frac{L^2}{4D}>0,
    \label{eqn:diffusionTimescale}
\end{equation}
where $L>0$ is the shortest distance from the searcher's initial position(s) to the target(s), and $D$ is the diffusivity of the searcher \cite{lawley2020uni}. The prefactor $A$ and power $p$ in \eqref{eq:exponential} depend on details of the geometry of the search space \cite{lawley2020dist}.

\subsection{Time inhomogeneous immigration rates}
\label{sec:TIimmigrationrates}
In the time inhomogeneous immigration process (TII process), particles immigrate into the system according to a time inhomogeneous Poisson point process (PPP) with rate $\lambda u(t)$ at time $t$. This means that for small $\Delta t$, the number of searchers immigrating in on the time interval $[t, t+\Delta t]$ is Poisson distributed with mean $\lambda u(t) \Delta t$. As such, the number of searchers that arrive in time interval $[t_1, t_2)$ is Poisson distributed with mean $\int_{t_1}^{t_2} \lambda u(t)dt$. Furthermore, the number of searchers arriving in disjoint time intervals are independent from each other. We note that in the constant rate case $u(t) = 1$, the times between immigrations are exponential random variables with rate $\lambda$.

Note that $P(\sigma_1 >t)$ is the probability that no searchers immigrate in the interval $[0, t)$, which is the probability that a Poisson random variable is $0$,
$$
P(\sigma_1 >t) = \exp\left[-\int_0^t \lambda u(s) ds\right].
$$
Similarly, $P(\sigma_k >t)$ is the probability that fewer than $k$ searchers immigrate in the interval $[0, t)$, which is
$$
P(\sigma_k >t) = \exp\left[-\int_0^t \lambda u(s) ds\right] \sum_{i=0}^{k-1}\frac{\left(\int_0^t \lambda u(s) ds\right)^i}{i!}.
$$
Independence on disjoint intervals for PPPs also tells us that
$$
P(\sigma_k >t + \sigma_{k-1}|\sigma_{k-1}) = \exp\left[-\int_{\sigma_{k-1}}^{t+\sigma_{k-1}} \lambda u(s) ds\right], \quad k\geq 2.
$$

\subsection{Yule immigration rates}
\label{sec:YIimmigrationrates}
In the Yule immigration (YI) process, there is an initial searcher (i.e.\ $\sigma_1 = 0$), and if there are $N(t)$ searchers at time $t$, then searchers immigrate in at rate $\lambda N(t)$. In short, the searcher population grows according to a pure birth process with rate $\lambda$, which is also referred to as a Yule process. This implies
\begin{equation}
    P(\sigma_k >t + \sigma_{k-1}) = \exp\left[-(k-1)\lambda t\right], \quad k\geq 2.
    \label{eqn:YuleInterarrival}
\end{equation}
It is well known that in a Yule process with growth rate $\lambda$ (see Example 4.10 in \cite{durrett2016essentials}), if the initial population size is $1$, then the increase in population size by $t$ has a Geometric distribution, specifically
\begin{equation}
    P(N(t) = k) = P(\geo(e^{-\lambda t}) = k-1) = e^{-\lambda t}(1-e^{-\lambda t})^{k-1}, \quad k\geq 1.
    \label{eqn:YulePopSize}
\end{equation}
As such,
$$
P(\sigma_k>t) = \sum_{i = 1}^{k-1} P(N(t) = i) = 1-(1-e^{-\lambda t})^k, \quad k\geq 1.
$$

\section{Theoretical results - time inhomogeneous immigration}
\label{sec:TIresults}
In section \ref{sec:STIequality}, we express the survival probability of the $k$th fastest searcher for the TII process in terms of the survival probability of a single searcher. Using this representation, we find in section \ref{sec:TIlimitDist} the limiting distribution of $T_k$ for large $\lambda$. Section \ref{sec:TIlimitMom} uses the limiting distribution to find the moments of $T_k$ for large $\lambda$.

\subsection{An exact survival probability formula}
\label{sec:STIequality}
Let $S_{TI, k}(t)=P(T_k>t)$ denote the probability that fewer than $k$ searchers have found the target by time $t$. To find an exact formula for $S_{TI, k}(t)$, we filter our PPP for immigration times to find a PPP for searchers who will arrive at the target by time $t$, who we will refer to as successful searchers. Since a searcher that arrives at time $s$ will have probability $1-S(t-s)$ of finding the target by time $t$, the PPP for successful searchers has rate $\lambda u(s)(1-S(t-s))$ at time $s$. This implies the number of successful searchers on $[0, t)$ is Poisson distributed with mean $\lambda I(t)$, where
$$
I(t) = \int_0^t u(s)(1-S(t-s)) ds = \int_0^t u(t-s)(1-S(s)) ds.
$$ 
Since $S_{TI, k}(t)$ is the probability of fewer than $k$ successful searcher on $[0, t)$,
\begin{proposition} 
  \begin{equation}
    S_{TI, k}(t) =\exp\left[-\lambda I(t)\right]\sum_{i = 0}^{k-1}\frac{\left(\lambda I(t)\right)^i}{i!}.
    \label{eqn:SIequalityInhomo}
  \end{equation}
  \label{prop:SIequalityInhomo}
\end{proposition}

\subsection{Convergence in distribution}
\label{sec:TIlimitDist}
In the large $\lambda$ limit, we show that $(T_k-b_\lambda)/a_\lambda$ converges in distribution, or equivalently,
\begin{theorem}
    In the TII process, let $I(t) = \int_0^t u(t-s)(1-S(s))ds$ with $I(t) \sim A_0t^{p_0}$ where $A_0, p_0>0$ or $I(t) \sim A_0t^{p_0}e^{-C_0/t}$ where $A_0>0, p_0\in\R,$ and $C_0>0$. Then
    $$
    \lim_{\lambda\rightarrow \infty} S_{TI, k}(a_{\lambda} x + b_{\lambda}) = \begin{cases}
  \frac{1}{(k-1)!}\Gamma\left(k, x^{p_0}\right) & \text{if }I(t) \sim A_0t^{p_0} \\
  \frac{1}{(k-1)!}\Gamma\left(k, e^x\right) & \text{if }I(t) \sim A_0t^{p_0}e^{-C_0/t}
\end{cases}
    $$
    where 
    $$
    (a_\lambda, b_\lambda) = \begin{cases}
  \left(\left(A_0\lambda\right)^{-1/p_0},\quad 0\right) & \text{if }I(t) \sim A_0t^{p_0} \\
  \left(\frac{C_0}{(\ln(C_0\lambda))^2},\quad \frac{C_0}{\ln (C_0\lambda)}
        +\frac{C_0p_0\ln(\ln(C_0\lambda))}{(\ln (C_0\lambda))^2}
        -\frac{C_0\ln(A_0C_0^{p_0})}{(\ln (C_0\lambda))^2}\right) & \text{if }I(t) \sim A_0t^{p_0}e^{-C_0/t}
\end{cases}
    $$
    and $\Gamma(r,z)$ denotes the upper incomplete gamma function, $\Gamma(r,z)=\int_z^\infty s^{r-1}e^{-s}\,\dd s$.
    \label{thm:TIlimitDist}
\end{theorem}
When $k=1$, we recover the Gumbel and Weibull distributions, which are the limiting distributions from the Fisher-Tippett-Gnedenko Theorem \cite{colesbook}. This might be surprising since as noted in \eqref{eq:collection} the search times $\{\tau_n + \sigma_n\}_n$ are not iid and therefore Fisher-Tippett-Gnedenko is not immediately applicable. The reason why results for TII are reminiscent of results from previous work on large $N$ many searchers can be explained through our proof methodology, which couples the TII process with the large $N$ many searchers process. Details are in the Appendix.

The remainder of this subsection discusses how Theorem \ref{thm:TIlimitDist} applies when $u(t) \sim \alpha_nt^n$ for nonnegative integer $n$ and positive $\alpha_n$. When $1-S(t) \sim At^p$, one can integrate
$$
I(t) \sim \int_0^t \alpha_n(t-s)^n As^p ds = \alpha_n\sum_{j=0}^n {n \choose j}t^j\int_0^t (-s)^{n-j} As^p ds = \frac{\Gamma(n+1)\Gamma(p+1)}{\Gamma(n+p+2)}A\alpha_nt^{n+p+1},
$$
and therefore we can apply Theorem \ref{thm:TIlimitDist} with 
$$
A_0 = \frac{\Gamma(n+1)\Gamma(p+1)}{\Gamma(n+p+2)}A\alpha_n, \qquad p_0 = n+p+1.
$$
When $1-S(t) \sim At^pe^{-C/t}$, we are unable to directly integrate, but we can still show that
\begin{equation}
    I(t) \sim \frac{A\alpha_nn!}{C^{n+1}}t^{p+2n+2}e^{-C/t},
    \label{eqn:shortTimeIt}
\end{equation}
and therefore we can apply Theorem \ref{thm:TIlimitDist} with 
$$
A_0 = \frac{A\alpha_nn!}{C^{n+1}}, \qquad p_0 = p+2n+2 \qquad C_0 = C.
$$
The details of this derivation are in Appendix \ref{app:TInumerics}.

\subsection{Convergence of moments}
\label{sec:TIlimitMom}
Since we showed that $(T_k-b_\lambda)/a_\lambda$ converges in distribution as $\lambda\rightarrow \infty$, one might conjecture that the moments of $(T_k-b_\lambda)/a_\lambda$ converge to the moments of the limiting distribution. We prove in the Appendix that this is the case.

\begin{theorem}
    In the TII process, let $I(t), a_\lambda,$ and $b_\lambda$ be defined as in Theorem \ref{thm:TIlimitDist}. For any integer $m \geq 1$, if $\int_0^\infty u(t) dt = \infty$, then
    $$
    \lim_{\lambda \rightarrow \infty}\E\left[\left(\frac{T_{k} - b_\lambda}{a_\lambda}\right)^m\right] = \begin{cases}
      \frac{1}{(k-1)!}\Gamma\left(k+\frac{m}{p_0}\right) & \text{if }I(t) \sim A_0t^{p_0} \\
      \frac{1}{(k-1)!}\frac{d^m}{dt^m}\Gamma\left(k+t\right)\bigg\rvert_{t=0} & \text{if }I(t) \sim A_0t^{p_0}e^{-C_0/t}.
    \end{cases}
    $$
    \label{thm:TIlimitMom}
\end{theorem}

The additional condition $\int_0^\infty u(t) dt = \infty$ is used to ensure FPTs are finite. For example, if $u(t) =1$ on $[0, 1]$ and $u(t) = 0$ on $[1, \infty)$, then for any $\lambda$ there is a nonzero chance that no searchers ever immigrate in.

A natural consequence of Theorem \ref{thm:TIlimitDist} is that, plugging in $m=1$ or $2$, we find 
\begin{align*}
    \lim_{\lambda \rightarrow \infty}\E\left[\left(\frac{T_k-b_\lambda}{a_\lambda}\right)\right] &= \begin{cases}
      \frac{1}{(k-1)!}\Gamma\left(k+\frac{1}{p_0}\right) & \text{if }I(t) \sim A_0t^{p_0} \\
      -\gamma + \sum_{i=1}^{k-1}\frac{1}{i} & \text{if }I(t) \sim A_0t^{p_0}e^{-C_0/t}
    \end{cases} \\
    \lim_{\lambda \rightarrow \infty}\E\left[\left(\frac{T_k-b_\lambda}{a_\lambda}\right)^2\right] &= \begin{cases}
      \frac{1}{(k-1)!}\Gamma\left(k+\frac{2}{p_0}\right) & \text{if }I(t) \sim A_0t^{p_0} \\
      \left(-\gamma + \sum_{i=1}^{k-1}\frac{1}{i}\right)^2+\frac{\pi^2}{6}- \sum_{i=1}^{k-1}\frac{1}{i^2} & \text{if }I(t) \sim A_0t^{p_0}e^{-C_0/t}.
    \end{cases}
\end{align*}
Therefore, the mean and variance of the $k$th fastest search time have the following asymptotic expansions as $\lambda\to\infty$,
\begin{align}
    \E[T_k] &= \begin{cases}
      \frac{(A_0\lambda)^{-1/p_0}}{(k-1)!}\Gamma\left(k+\frac{1}{p_0}\right) + o(\lambda^{-1/p_0}) & \text{if }I(t) \sim A_0t^{p_0} \\
      \left(-\gamma + \sum_{i=1}^{k-1}\frac{1}{i}\right)a_\lambda + b_\lambda + o(1/\ln(\lambda)^2) & \text{if }I(t) \sim A_0t^{p_0}e^{-C_0/t}
    \end{cases} \label{eqn:TIkthmean}\\
    \Var[T_k^2] &= \begin{cases}
      \left(\frac{1}{(k-1)!}\Gamma\left(k+\frac{2}{p_0}\right) -\frac{1}{(k-1)!^2}\Gamma\left(k+\frac{1}{p_0}\right)^2\right)(A_0\lambda)^{-2/p_0} + o(\lambda^{-2/p_0}) & \text{if }I(t) \sim A_0t^{p_0} \\
      \left(\frac{\pi^2}{6}- \sum_{i=1}^{k-1}\frac{1}{i^2}\right)a_\lambda^2 + o(1/\ln(\lambda)^4) & \text{if }I(t) \sim A_0t^{p_0}e^{-C_0/t}
    \end{cases} \nonumber
\end{align}

\section{Theoretical results - Yule immigration}
\label{sec:YIresults}
In section \ref{sec:SYIequality}, we express the survival probability of the $k$th fastest searcher for the YI process in terms of the survival probability of a single searcher. Using this representation, we find in section \ref{sec:YIlimitDist} the limiting distribution of $T_k$ for large $\lambda$. Section \ref{sec:YIlimitMom} uses the limiting distribution to find the moments of $T_k$.

\subsection{An exact survival probability formula}
\label{sec:SYIequality}
Let $S_{YI, k}(t)=P(T_k>t)$ denote the probability that fewer than $k$ searchers have found the target by time $t$. If we condition on $N(t) = n$, then we can generate $\{\sigma_i\}_{2\leq i \leq n}$, the times at which searchers immigrated in after time $0$, by taking $n$ iid samples from a truncated exponential random variable, 
\begin{equation}
    P(\sigma_i'\leq a) = \frac{e^{\lambda a}-1}{e^{\lambda t}-1}, \quad \frac{d}{da} P(\sigma_i'\leq a) = \frac{\lambda e^{\lambda a}}{e^{\lambda t}-1} = \frac{\lambda e^{\lambda (a-t)}}{1-e^{-\lambda t}}, \quad \sigma_i'\in[0, t],
    \label{eqn:YIindividual}
\end{equation}
and reordering $\{\sigma_i'\}$ from smallest to largest. The proof of \eqref{eqn:YIindividual} is in the appendix. As such, when conditioned on $N(t)$ at time $t$, the probability a given searcher that immigrated in after $t=0$ is successful (where again ``successful'' means that the searcher finds the target before time $t$), is
$$
\int_0^t (1-S(t-s))\left(\frac{d}{ds} P(\sigma_i'\leq s)\right)ds = \int_0^t \frac{(1-S(s))\lambda e^{-\lambda s}}{1-e^{-\lambda t}} ds,
$$
and the total number of successful searchers that immigrated in on time interval $(0, t]$ is Binomial distributed,
\begin{equation}
\theta = \binom\left(N(t)-1, \int_0^t \frac{(1-S(s))\lambda e^{-\lambda s}}{1-e^{-\lambda t}} ds\right).
\label{eqn:theta}
\end{equation}
This can be simplified further. Recall from \eqref{eqn:YulePopSize} that $N(t)-1$ is geometric distributed with success probability $e^{-\lambda t}$. Then $\theta$ is a binomial with geometrically distributed many trials. It is well known (and we prove in the appendix for the sake of completeness) that a binomial random variable with a geometrically distributed number of trials is geometric,
\begin{equation}
    \binom\left(\geo(p_1), p_2\right) = \geo\left(\frac{p_1}{p_1 + p_2(1-p_1)}\right).
    \label{eqn:binomGeo}
\end{equation}
Using \eqref{eqn:binomGeo} with $p_1 = e^{-\lambda t}$ and $p_2 = \int_0^t (1-S(s))\lambda e^{-\lambda s} ds/(1-e^{-\lambda t})$ yields
\begin{proposition} 
   For YI, the number of successful searchers that immigrated in times $(0, t]$ is geometric with support on $\{0, 1, 2, \ldots\}$ and success probability
   $$
   p(t) = \frac{1}{1 + \lambda e^{\lambda t}\int_0^t (1-S(s))e^{-\lambda s} ds}.
   $$
   Furthermore,
   $$
   S_{YI, k}(t) = S(t)[1-(1-p(t))^k] + (1-S(t))[1-(1-p(t))^{k-1}].
   $$
   \label{prop:SYIexact}
\end{proposition}

\subsection{Convergence in distribution for large $\lambda$}
\label{sec:YIlimitDist}
In the large $\lambda$ limit, we show that $(T_k-b_\lambda)/a_\lambda$ converges in distribution, or equivalently,
\begin{theorem}
   In the YI process,
   $$
   \lim_{\lambda \rightarrow\infty}S_{YI, k}(a_\lambda x + b_\lambda) = 1-\Big(1-\frac{1}{1+e^x}\Big)^k,
   $$
   where
   $$
   (a_\lambda, b_\lambda) = \begin{cases}
       \left(\frac{1}{\lambda}, \frac{1}{\lambda}\ln\left(\frac{\lambda^p}{A\Gamma(p+1)}\right)\right) & \text{if }1-S(t) \sim At^p \\
       \left(\frac{1}{\lambda}, \frac{2\sqrt{C\lambda}+\frac{2p-1}{4}\ln(C\lambda) - \ln(AC^p\sqrt{\pi})}{\lambda}\right) &\text{if }1-S(t) \sim At^pe^{-C/t}.
   \end{cases}
   $$
\label{thm:YIlimitDist}
\end{theorem}
Setting $S_{YI, k}(a_\lambda x + b_\lambda) = 1/2$ indicates that the median $k$th passage time is
\begin{equation}
    \textup{median}(T_k)
    =b_\lambda-a_\lambda \ln(2^{1/k}-1) + o(1/\lambda)\quad\text{as }\lambda\to\infty.
    \label{eqn:YImedian}
\end{equation}

Theorem \ref{thm:YIlimitDist} differs from analogous results for large $N$, constant immigration, and TII in three key ways. The first observation, and least surprising, is that $a_\lambda$ and $b_\lambda$ now shrink in $O(1/\lambda)$ rather than $O(\ln(\lambda))$ or $O(\ln(N))$. The second observation is that in either case of short time behavior, we find the same limiting distribution. The third observation is that the limiting distribution of the fastest FPT ($k=1$) is not Weibull or Gumbel; rather, it is a logistic distribution, which can be expressed as the difference between two independent standard Gumbel distributed random variables.

We now give a heuristic argument to understand the second and third observations. It is well known (see Example 4.10 in \cite{durrett2016essentials}) that the population size $N(t)$ in a Yule process satisfies $e^{-\lambda t} N(t) \rightarrow V$ as $t\rightarrow \infty$ where $V$ is an exponential random variable with rate $1$. We might then approximate $N(t)$ as
$$
Ve^{\lambda t} = e^{\lambda (t- \ln(V)/\lambda)}.
$$ 
This suggests that one can also examine the fastest FPT of YI by instead examining a time inhomogeneous immigration process where the immigration rate is $u(t) = Ve^{\lambda t}$. Indeed, one can verify that
$$
S_{YI, 1}(t) = p(t) = \E\left[e^{-\lambda\int_0^t Ve^{\lambda (t-s)}(1-S(s))ds}\right].
$$
As such, we can view differing values of $V$ as time shifts of size $\ln(V)/\lambda$ compared to the process where $V=1$. Letting $\tau_\lambda$ denote the fastest FPT when $V=1$, then
$$
S_{YI, 1}(a_\lambda x + b_\lambda) \approx P\left(\tau_\lambda - \frac{1}{\lambda}\ln(V) > a_\lambda x + b_\lambda\right) = P\left(\lambda\tau_\lambda - \ln(V) > x+\lambda b_\lambda\right).
$$
It is straightforward to verify that
$$
P(\ln(V)>x) = P(V>e^x) = e^{-e^x},
$$
and therefore $X_1:=\ln(V)$ is a Gumbel distributed random variable. Furthermore, one can also verify that
$$
P(\tau_\lambda > a_\lambda x + b_\lambda) = \exp\left[-\lambda\int_0^{a_\lambda x + b_\lambda} e^{\lambda (a_\lambda x + b_\lambda -s)}(1-S(s))ds\right] \rightarrow e^{-e^x},
$$
as $\lambda \rightarrow \infty$ and therefore $\tau_\lambda \approx X_2/\lambda + b_\lambda$ where $X_2$ is Gumbel. Plugging back yields
$$
S_{YI, 1}(a_\lambda x + b_\lambda) \approx P(X_1-X_2 > x).
$$
Therefore, the limiting distribution of the fastest FPT of YI is the difference of two Gumbel distributions, with one Gumbel representing the hitting time in an average scenario where population grows deterministically, and the other representing a time shift from the inherent stochasticity of population growth.

\subsection{Convergence of moments for large $\lambda$}
\label{sec:YIlimitMom}
Similar to the TII case, one might expect that the moments of $(T_k-b_\lambda)/a_\lambda$ converge to the moments of the limiting distribution. One can compute the moment generating function of the limiting distribution by taking a computer algebra software like Mathematica to find
$$
\int_{-\infty}^\infty e^{tx} \left[\frac{d}{dx}\left(1-\left(1-\frac{1}{1+e^x}\right)^k\right)\right] dx = \frac{1}{(k-1)!}\Gamma\left(k+t\right)\Gamma\left(1-t\right),
$$
where $\Gamma(x)$ is the Gamma function. By showing uniform integrability of $((T_k-b_\lambda)/a_\lambda)^m$, it follows that
\begin{theorem}
    In the YI process, for any integer $m \geq 1$,
    $$
    \lim_{\lambda \rightarrow \infty}\E\left[\left(\frac{T_k-b_\lambda}{a_\lambda}\right)^m\right] = \frac{1}{(k-1)!}\frac{d^m}{dt^m}\Gamma\left(k+t\right)\Gamma\left(1-t\right)\bigg\rvert_{t=0}.
    $$
    \label{thm:YImoments}
\end{theorem}
\noindent Notably, plugging in $m=1$ or $2$, we find that
$$
\lim_{\lambda \rightarrow \infty}\E\left[\left(\frac{T_k-b_\lambda}{a_\lambda}\right)\right] = \sum_{i=1}^{k-1}\frac{1}{i}, \qquad \lim_{\lambda \rightarrow \infty}\E\left[\left(\frac{T_k-b_\lambda}{a_\lambda}\right)^2\right] = \left(\sum_{i=1}^{k-1}\frac{1}{i}\right)^2 + \frac{\pi^2}{3}- \left(\sum_{i=1}^{k-1}\frac{1}{i^2}\right), 
$$
and therefore
\begin{align}
    \E[T_k] &= \left(\sum_{i=1}^{k-1}\frac{1}{i}\right)\frac{1}{\lambda} + b_\lambda + o(1/\lambda) \label{eqn:YIkthmean}\\
    \Var[T_k^2] &= \left(\frac{\pi^2}{3}-\sum_{i=1}^{k-1}\frac{1}{i^2}\right)\frac{1}{\lambda^2} + o(1/\lambda^2) \nonumber
\end{align}
by plugging in $a_\lambda = 1/\lambda$ and using $b_\lambda$ from Theorem \ref{thm:YIlimitDist} depending on the short time behavior of a searcher.

\subsection{Comparisons with branching models}\label{sec:comparisontobranching}
In contrast to our YI model, where all added searchers start from the same distribution of starting locations, one could also consider the case where, each time a new searcher is added, an existing searcher is chosen uniformly at random to be a parent and the new searcher is added at the parent's location. In this section, we compare results for YI with previous work for a branching process (BP) on a network modeling cancer as well as branching brownian motion (BBM).

\subsubsection{Comparison to BP}\label{sec:branchingnetwork}

Refs.~\cite{durrett2015branching, durrett2010evolution} modeled cancer evolution by starting with a single cancer cell of type $0$. Cells of type $i$ reproduce according to a branching process with net fitness $\lambda_i < \lambda_{i+1}$ and mutate from type $i$ to type $i+1$ with rate $\mu_{i+1}$. When a mutation occurs, the type $i+1$ population increases by 1 and the type $i$ population stays the same. In the small mutation limit, they found expressions for the first time a cell with $n$ mutations appears, which we call $T_1^{BP, n}$, as well as the asymptotic population size of cells of type $n$. Ref.~\cite{nicholson2023sequential} extended their results so that fitnesses do not need to be increasing. When $\lambda_i = \lambda$ does not vary with cell type, then the process is a BP on a linear chain and, by equations 4 and 5 in \cite{nicholson2023sequential},
$$
P\left[T_1^{BP, n} > \frac{x}{\lambda} + \frac{1}{\lambda}\ln\left(\frac{\lambda^n}{\mu_1\mu_2\ldots \mu_{k}}\frac{(n-1)!}{\ln(\lambda/\mu_n)^{n-1}}\right)\right] \rightarrow \frac{1}{1+e^x}
$$
as the mutation rates approach $0$. By a units argument, rather than letting the mutation rates go to $0$, we can instead let $\lambda \rightarrow \infty$ and find the same result. 

Now, consider the corresponding problem for YI process - the FPT of the YI process on a network that is a linear chain, searchers move from compartment $i$ to $i+1$ with rate $\mu_{i+1}$, and the target is compartment $n$. Recall from Section \ref{sec:model} that the FPT for one searcher to reach compartment $n$ has short time behavior $1-S(t) \sim At^p$ where $p=n$ is the length of the shortest path and $A = \mu_1 \ldots \mu_n /n!$ is the product of rates along the shortest path divided by $p!$. Applying Theorem \ref{thm:YIlimitDist}, we then find
$$
\lim_{\lambda \rightarrow \infty} S_{YI, 1}\left(\frac{x}{\lambda} + \frac{1}{\lambda}\ln\left(\frac{\lambda^n}{\mu_1 \ldots \mu_n}\right)\right) = \frac{1}{1+e^x}.
$$
The two survival probabilities for the BP and YI processes are very close in the large $\lambda$ limit, with the same limiting distribution. Even though the BP is always faster than YI in this case (the new searcher is always added to the farthest node from the target in YI), the difference is only a slight shift in the median.

\subsubsection{Comparison to BBM}

In the case of BBMs, seminal works \cite{bramson1978maximal, bramson1983convergence, lalley1987conditional} studied the rightmost particle at time $t$ of a BBM taking place on $\mathbbm{R}$. They found that the probability the rightmost particle at time $t$ of a BBM is larger than $x$ is the solution to a Fisher-KPP equation, and that it can be expressed as a Gumbel shifted by a constant and the limit of a derivative martingale. For large time $t$, the rightmost particle at time $t$ is around
$$
m_t = \sqrt{2}t - \frac{3}{2\sqrt{2}}\ln(t).
$$

Although the rightmost particle being past $x$ at time $t$ is not the same as the probability a BBM has reached $x$ before time $t$, one can expect them to be similar (this is akin to the so-called ``non-backtracking approximation'' in \cite{hass2024first, hass2024extreme}). Recent work by Kim et al in 2023 \cite{kim2023maximum} have generalized results from $\mathbbm{R}$ to $\mathbbm{R}^d$. 

Analogous to the network example in section~\ref{sec:branchingnetwork}, let us examine the FPT of a BBM on the real line starting at $0$ searching for a target at $L>0$. Unlike in network case, using a BBM rather than YI can speed up searches by a factor of 2. The remainder of this section proves this claim. 

Let $M_{\lambda, D}(t)$ denote the maximum ever achieved for a BBM with diffusivity $D$ and branching rate $\lambda$ up to time $t$. Note that a standard Brownian motion on $\mathbb{R}$ has diffusivity $1/2$. Theorem 10 from \cite{berestycki2017branching} showed that for some constant $C^*$,
$$
\lim_{t\rightarrow \infty}P(M_{1, 1/2}(t) < m_t+C^*) = \frac{1}{2},
$$
which naturally implies that
$$
\lim_{\lambda\rightarrow \infty}P(M_{1, 1/2}(\lambda t) < m_{\lambda t}+C^*) = \frac{1}{2}.
$$
To relate back to the general case of $\lambda$ branching and $D$ diffusivity, note that by Brownian scaling,
\begin{align*}
    P(M_{\lambda, D}(t) < L) &= P\bigg(\sqrt{\frac{\lambda}{2D}}M_{\lambda, D}(t) < \sqrt{\frac{\lambda}{2D}}L\bigg) \\
    &= P\bigg(\sqrt{\lambda}M_{\lambda, 1/2}(t) < \sqrt{\frac{\lambda}{2D}}L\bigg) \\
    &= P\bigg(M_{1, 1/2}( \lambda t) < \sqrt{\frac{\lambda}{2D}}L\bigg).
\end{align*}
Putting the two equations together, we find that for large $\lambda$, the median of the FPT to $L>0$ for a BBM with diffusivity $D$ and branching rate $\lambda$ is when $t$ solves
$$
m_{\lambda t}+C^* = \sqrt{\frac{\lambda}{2D}}L.
$$
Taking the leading order estimate $\sqrt{2}\lambda t$ of $m_{\lambda t}+C^*$, our leading order estimate of the median is then
$$
t=\frac{L}{2\sqrt{D\lambda}}.
$$
For the median time for YI, using \eqref{eqn:YImedian} with $k=1$ yields the leading term $2\sqrt{C/\lambda}$. Applying $C=L^2/(4D)$ from \eqref{eqn:diffusionTimescale} then yields $L/\sqrt{D\lambda}$, which is twice as large as the leading order estimate for BBM. 

\section{Numerics}\label{sec:numerics}
As in \cite{tung2025first}, we examine the speed of convergence for both the distribution and the mean under three scenarios - diffusion on $\mathbb{R}$, diffusion on $\mathbb{R}^3$, and a random walk on a discrete grid.

\subsection{Diffusive search in one space dimension}\label{sec:oned}
In the first example, searchers immigrate to $L>0$ and undergo Brownian motion on $\mathbb{R}$ with diffusivity $D>0$ until they reach the target, the origin. It is well known \cite{carslaw1959} that the survival probability for a single searcher is
$$
S(t)
=P(\tau>t)
= \text{erf}\bigg(\sqrt{\frac{L^2}{4Dt}}\bigg),
$$
where erf is the error function. It follows that
$$
1-S(t) \sim At^pe^{-C/t}, \qquad A = \sqrt{\frac{4D}{L^2\pi}}, \qquad p= \frac{1}{2}, \qquad C = \frac{L^2}{4D}.
$$

\subsection{Diffusive escape in three space dimensions}
In the second example, searchers immigrate to the origin and undergo Brownian motion on $\mathbb{R}^3$ with diffusivity $D>0$ until they reach the target, a sphere of radius $L>0$. It is well known \cite{carslaw1959} that the survival probability for a single searcher is
$$
S(t) = 1-\sqrt{\frac{4L^2}{\pi D t}}\sum_{j=0}^\infty \exp\left[-\frac{L^2}{4Dt}(2j+1)^2\right],
$$
and therefore
$$
1-S(t) \sim At^pe^{-C/t}, \qquad A = \sqrt{\frac{4L^2}{D\pi}}, \qquad p= -\frac{1}{2},\qquad C = \frac{L^2}{4D}.
$$

\subsection{Search on a discrete network}
In the third example, searchers immigrate to the upper left corner corner of a $5 \times 5$ discrete square grid and move with rate $1$ between neighboring vertices until they find the target located one down and two to the right. The survival probability for a single searcher can be expressed as the product of elementary row vectors and the matrix exponential of an appropriately constructed rate matrix \cite{norris1998}. It follows from Section \ref{sec:model} (see Proposition~1 in \cite{lawley2020networks} for details) that the short time behavior is 
$$
1-S(t) \sim At^p,\qquad A = \frac{1}{2}, \qquad p = 3.
$$

\subsection{Comparison to numerics}
For these three examples, Figure \ref{fig:numericsLimitingDist} compares the probability densities of the fastest FPT $T_1$ to their limits obtained in Theorems \ref{thm:TIlimitDist} and \ref{thm:YIlimitDist}. Similarly, Figure \ref{fig:numericsMeanError} compares the means of the fastest FPT $T_1$ to their limits obtained in Theorems \ref{thm:TIlimitMom} and \ref{thm:YImoments}. 

Note that for faster convergence, in the case of TII where $1-S(t) \sim A_0t^{p_0}e^{-C_0/t}$, we do not use the $(a_\lambda, b_\lambda)$ pair given in \ref{thm:TIlimitDist} and instead use the pair
$$
(a_\lambda, b_\lambda) = \left(\frac{C_0}{p_0^2 W(1+W)} , \frac{C_0}{p_0 W} \right), \qquad 
        W
        =\begin{cases}
            W_0\Big[\frac{C_0}{p_0}\big(A_0\lambda\big)^{\frac{1}{p_0}}\Big] & \text{if }p_0>0,\\
            W_{-1}\Big[\frac{C_0}{p_0}\big(A_0\lambda\big)^{\frac{1}{p_0}}\Big]& \text{if }p_0<0,
        \end{cases}
$$
where $W_0(z)$ denotes the principal branch of the LambertW function and $W_{-1}(z)$ denotes the lower branch \cite{corless1996}. Our alternative pair works because, by elementary properties of convergence,
\begin{align}\label{eq:elem}
    \lim_{\lambda\to\infty}\frac{a_{\lambda}'}{a_{\lambda}}
=1,\quad
\lim_{\lambda\to\infty}\frac{b_{\lambda}'-b_{\lambda}}{a_{\lambda}}
=0.
\end{align}
is a sufficient condition to showing $(a_{\lambda}', b_{\lambda}')$ suffices. We then use the asymptotics of the LambertW function \cite{corless1996},
\begin{align*}
    W_0(z)
    &=\ln z-\ln\ln z + o(1)\quad\text{as }z\to\infty,\\
    W_{-1}(z)
    &=\ln(-z)-\ln(-\ln(-z)) + o(1)\quad\text{as }z\to0-.
\end{align*}

\begin{figure}[ht]
    \centering
    \includegraphics[width=0.95\linewidth]{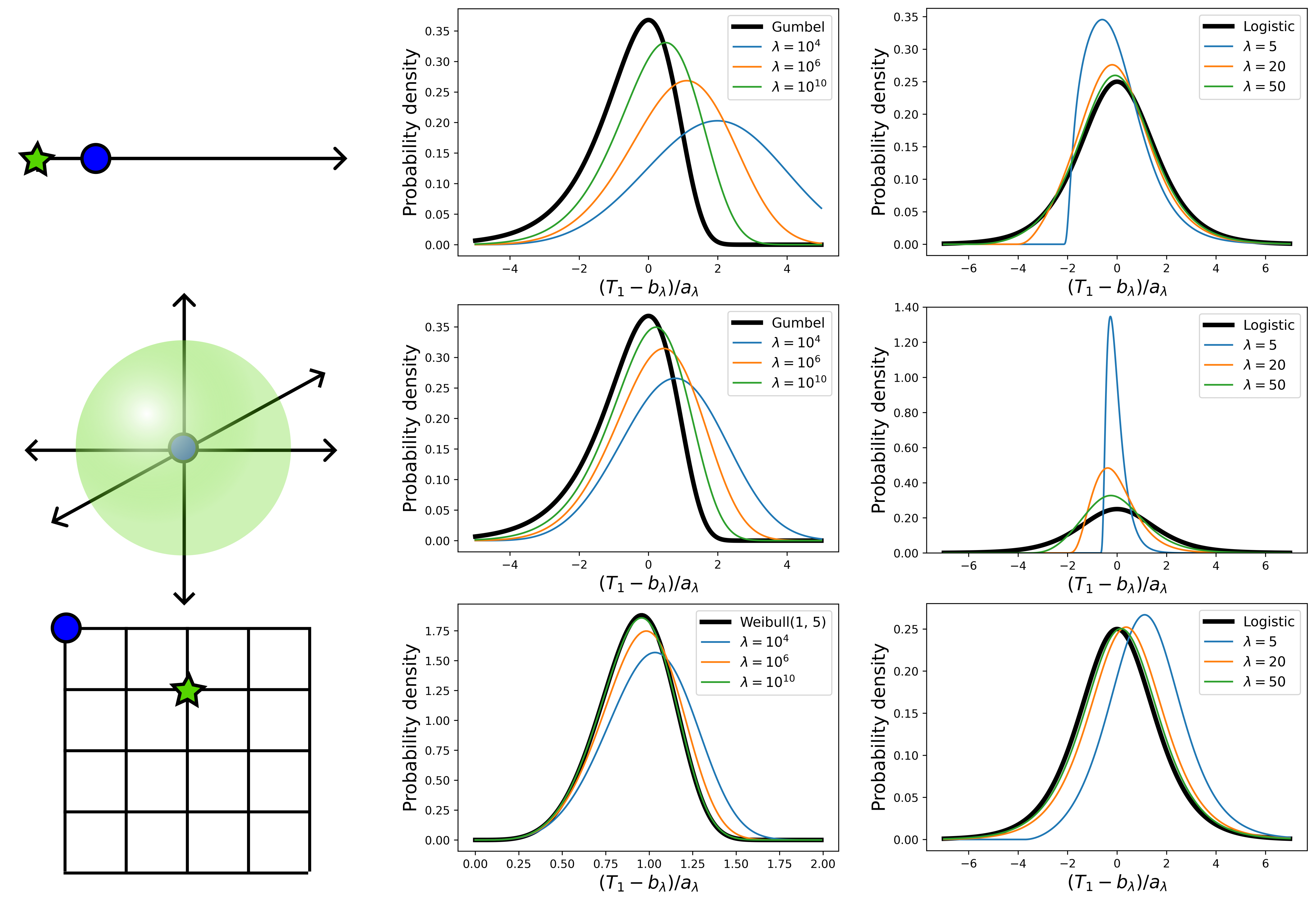}
    \caption{The first column depicts the three search examples used for numerics. Searchers start at the blue ball and their target is the green (the star, the sphere, and the star). The second column examines convergence in distribution for TII when $u(t)=t$ and shows corresponding plots of the probability densities of scaled first passage time $T_1$ for large $\lambda$ as well as the theoretical limiting distribution of Weibull or Gumbel given by Theorems \ref{thm:TIlimitDist}. The third column is similar to the second but instead examines the YI case. We take $L=D=1$ for the diffusion examples.}
    \label{fig:numericsLimitingDist}
\end{figure}

\begin{figure}[ht]
    \centering
    \includegraphics[width=0.95\linewidth]{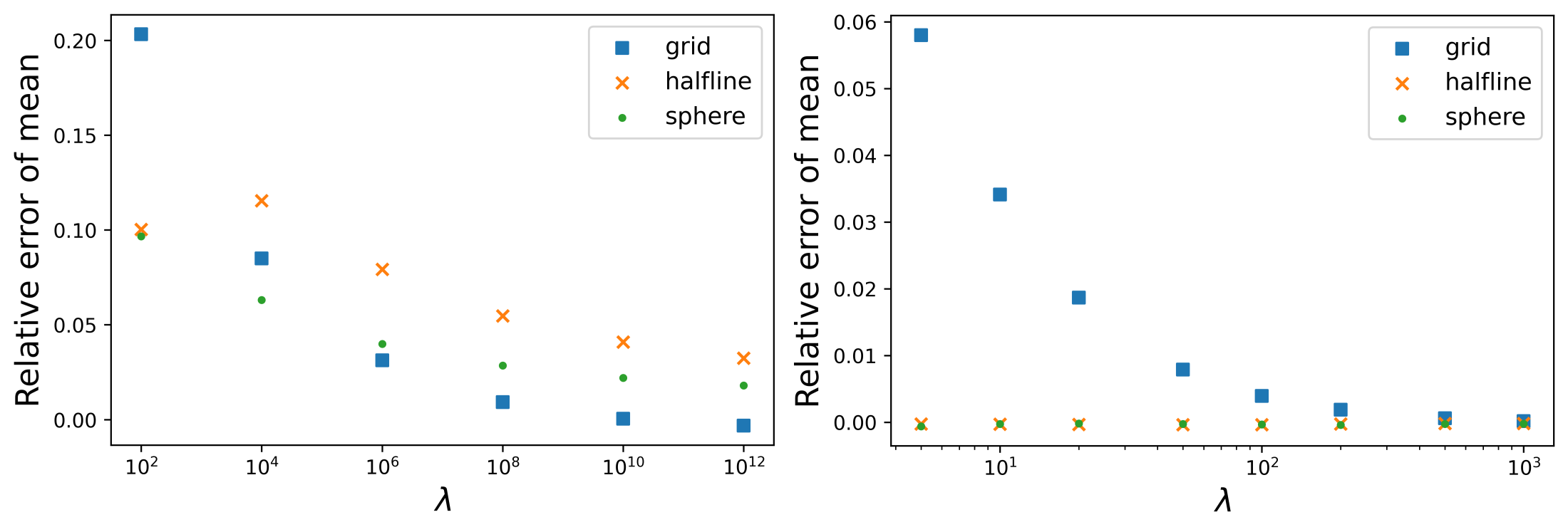}
    \caption{Both plots show the relative error in estimating the mean fastest FPT, $\E[T_1]$, as a function of $\lambda$ using \eqref{eqn:TIkthmean} and \eqref{eqn:YIkthmean} with $k=1$. The left plot is for TII when $u(t)=t$ and the right plot is for YI. We take $L=D=1$ for the diffusion examples.}
    \label{fig:numericsMeanError}
\end{figure}

\section{Discussion}\label{sec:discussion}

The majority of FPT theory has focused on a single random searcher \cite{redner_first_2014}. Though the analysis of the fastest searcher among many searchers dates back to the 1983 work of Weiss, Shuler, and Lindenberg \cite{weiss1983}, the broad biophysical relevance of such extreme FPTs has only been appreciated within the past decade \cite{grebenkov2024target}. Mathematical models of extreme FPTs have typically assumed that all $N\gg1$ searchers are initially present in the domain. 

In this paper, we studied extreme FPTs for searchers which progressively enter the domain over time under what has been termed an `immigration'' model \cite{tung2025first}.
We introduced two models of immigration (TII and YI) and studied extreme FPTs for each model. TII generalizes the model of Campos and Mendez \cite{campos2024} by allowing the immigration rate to change according to a given function of time. In YI, searchers enter at a rate proportional to the current number of searchers, and thus YI models systems in which searchers originate from births, replications, or infections.

In section~\ref{sec:numerics}, unlike for YI, we found slow convergence for TII in the large $\lambda$ limit. This is consistent with the recent work of Grebenkov, Metzler, and Oshanin, who introduced an interesting alternative model where searchers also progressively enter the domain over time under what has been termed an `injection'' model \cite{grebenkov2025fastest}. In their model, the total number of searchers $N\gg1$ is given, and these authors considered two choices for the times that the searchers enter the domain. They supposed that either (i) a given fraction $N \psi(t)\Delta t$ of searchers enters between time $t$ and $t+\Delta t$, or (ii) searchers enter at iid random times. Unlike our immigration models, (ii) keeps the FPTs of each searcher iid, making classical extreme value theory applicable. These authors found that scenarios (i) and (ii) both significantly slowed down the convergence of extreme FPTs to their limiting distributions in the large $N$ limit.

From the perspective of extreme value theory \cite{colesbook}, it is notable that we could obtain extreme value distributions exactly since the random variables are strongly correlated \cite{majumdar2020extreme}. Furthermore, the YI extreme distribution turned out not to follow the ``universal'' Frechet, Gumbel, and Weibull family of extreme value distributions. Instead, the YI extreme turned out to be the difference of two independent Gumbel random variables.

YI can also be considered as an alternative model of branching random walks and BBMs. The key difference is that in YI, the initial spatial location of new searchers always has the same probability distribution, whereas a new searcher in a branching process begins at a location that is uniformly distributed among the positions of existing searchers. This difference allowed us to prove general results for a wide variety of search processes undergoing YI. In contrast, the analysis of branching processes depends on the specifics of the search model (i.e.\ the methods used to study the branching processes in \cite{durrett2015branching, durrett2010evolution} differ from the methods used to study BBM \cite{berestycki2017branching, kim2023maximum}). Nevertheless, our results show that the extreme FPTs of YI differ only mildly from the analogous results for branching processes (see section~\ref{sec:comparisontobranching}).

\appendix
\section{A brief summary of large $N$ results} 
\label{app:largeN}
As the proofs of the theorems regarding TII in this paper rely heavily on results for FPTs with large $N$ initial searchers and no immigration, we provide a brief summary of relevant results here for convenience. For a more detailed treatment, see \cite{lawley2024competition}. 

Consider the case where there are $N$ initial iid searchers, each with survival probability $S(t)$. Then the survival probability of the system is the probability no searchers have found the target, which is
$$
S_N(t) = S(t)^N.
$$
Using classical extreme value theory and some algebra, one can show that under the survival probabilities that we've considered, 
\begin{theorem}[From references \cite{lawley2020dist} and \cite{madrid2020comp}]
    $$
    \lim_{\lambda \rightarrow \infty} S_N(a_N x + b_N)
        = 
    \begin{cases}
      \exp\left[-x^{p}\right] & \text{if }1-S(t) \sim At^p \\
      \exp\left[-e^x\right] & \text{if }1-S(t) \sim At^pe^{-C/t}
    \end{cases}
    $$
    where
    $$
    (a_N, b_N) = \begin{cases}
       ((AN)^{-1/p}, 0) & \text{if }1-S(t) \sim At^p \\
      \left(\frac{C}{(\ln N)^2}, \frac{C}{\ln N} + \frac{Cp \ln(\ln (N))}{(\ln N)^2} -\frac{C\ln(AC^p)}{(\ln N)^2}\right) & \text{if }1-S(t) \sim At^pe^{-C/t}.
    \end{cases}
    $$
    \label{thm:largeNFPT}
\end{theorem}

Analogous results hold for the $k$th fastest FPT. Because the searchers are iid, the number of successful searchers follows a Binomial distribution with $N$ trials and probability of success $1-S(t)$. Therefore, the probability that $k$ searchers have not found the target must then be 
$$
S_{N, k}(t) = \frac{n!}{(n-k)!(k-1)!}\int_0^{S(t)} u^{n-k}(1-u)^{k-1} du. 
$$
Once again, classical extreme value theory yields
\begin{theorem}[From references \cite{lawley2020dist} and \cite{madrid2020comp}]
    $$
    \lim_{\lambda \rightarrow \infty} S_{N, k}(a_N x + b_N)
        = 
    \begin{cases}
      \frac{1}{(k-1)!}\Gamma(k, x^{p+1}) & \text{if }1-S(t) \sim At^p \\
      \frac{1}{(k-1)!}\Gamma(k, e^x) & \text{if }1-S(t) \sim At^pe^{-C/t}.
    \end{cases}
    $$
    \label{thm:largeNkthFPT}
\end{theorem}

Letting $\mathcal{T}_{k, N}$ be the $k$th FPT for $N$ initial searchers, one can then show convergence of moments. Define
\begin{equation}
    \widetilde{\mathcal{T}_{k, N}} := \frac{\mathcal{T}_{k, N} - b_{N}}{a_N}.
    \label{eqn:TkNdef}
\end{equation}
Then
\begin{theorem}[From references \cite{lawley2020dist} and \cite{madrid2020comp}]
    If $\E[\mathcal{T}_N] < \infty$ for some $N\ge1$, then
    $$
    \lim_{N\rightarrow \infty}\E\left[\widetilde{\mathcal{T}_{k, N}}^m\right] = \begin{cases}
       \frac{1}{(k-1)!}\Gamma\left(k+\frac{m}{p+1}\right) & \text{if }1-S(t) \sim At^p \\
       \frac{1}{(k-1)!}\frac{d^m}{dt^m}\Gamma\left(k+t\right)\bigg\rvert_{t=0} & \text{if }1-S(t) \sim At^pe^{-C/t}.
    \end{cases}
    $$
    \label{thm:largeNFPTMom}
\end{theorem}
We further note that a fact that we will need in our proofs.
\begin{corollary}
    $\{\widetilde{\mathcal{T}_{k, N}}^m\}_{N}$ is uniformly integrable.
    \label{coro:largeNUI}
\end{corollary}
\begin{proof}
    In the case $1-S(t) \sim At^pe^{-C/t}$, this is already shown in the proof of Theorem \ref{thm:largeNFPTMom} (see the proof of Theorem 5 of \cite{lawley2020dist}). In the case $1-S(t) \sim At^p$, because we have convergence of moments, convergence in distribution, and integrability and nonnegativity of $\widetilde{\mathcal{T}_{k, N}}^m$, it follows that we have uniform integrability (see Theorem 3.6 in \cite{billingsley2013}).
\end{proof}

\section{Proofs of theorems for TII}
Before going to the theorem proofs, we first introduce some terminology, then prove two lemmas that recur through the proofs. There are two related processes that will be referenced in the proofs.
\begin{enumerate}
    \item TII process. This is the process discussed in this paper. The $k$th passage time is denoted by $T_{k}$.
    \item Large $N$ process. This process is discussed in Appendix \ref{app:largeN}. We will further assume each searcher has survival probability
    $$
    S_*(t):=\exp\left[-\int_0^t u(t-s) (1-S(s)) ds \right] = \exp\left[-I(t) \right]
    $$
    The $k$th passage time is denoted by $\mathcal{T}_{k}$.
\end{enumerate}
Letting $T$ denote the $k$th FPT of one of the above processes, we also use the tilde to denote
\begin{equation}
    \widetilde{T} := \frac{T - b_{\lambda}}{a_\lambda}.
\end{equation}

On to the two lemmas. Lemma \ref{lem:betterablambdas} gives an alternative set of $a_\lambda$ and $b_\lambda$ that will be more natural for the proofs and Lemma \ref{lem:naturalnumbers} lets limits for integer valued $\lambda$ carry over to real valued $\lambda$.
\begin{lemma}
    When $I(t) \sim A_0t^{p_0}e^{-C_0/t}$, instead of the $(a_\lambda, b_\lambda)$ pair given in Theorem \ref{thm:TIlimitDist}, one may use
    \begin{equation}
        (a_\lambda, b_\lambda) = \left(\frac{C_0}{(\ln(\lambda))^2},\quad \frac{C_0}{\ln (\lambda)}
        +\frac{C_0p_0\ln(\ln(\lambda))}{(\ln (\lambda))^2}
        -\frac{C_0\ln(A_0C_0^{p_0})}{(\ln (\lambda))^2}\right)
        \label{eqn:betterablambdas}
    \end{equation}
    \label{lem:betterablambdas}
\end{lemma}
\begin{proof}
    This follows directly from the condition \eqref{eq:elem}.
\end{proof}

\begin{lemma}
    Suppose the limit as $\lambda \rightarrow \infty$ of the expressions
    $$
    S_{TI, k}(a_\lambda x + b_\lambda), \quad \E[\widetilde{T_{k, \lambda}}^m]
    $$
    exist when $\lambda$ is restricted to the natural numbers. The limits exist and do not change when $\lambda$ is allowed to be any positive real number.
    \label{lem:naturalnumbers}
\end{lemma}
\begin{proof}
    By \eqref{eq:elem}, we can replace $(a_{\lambda}, b_{\lambda})$ with $(a_{\lfloor\lambda\rfloor}, b_{\lfloor\lambda\rfloor})$ or $(a_{\lceil\lambda\rceil}, b_{\lceil\lambda\rceil})$ and still have the same limits. Noting that faster immigration leads to faster FPTs and letting $S_{TI, k, \lambda}$ denote the survival probability for immigration rate $\lambda$,
    $$
    \lim_{\lambda\rightarrow\infty} S_{TI, k, \lfloor\lambda\rfloor}(a_{\lfloor\lambda\rfloor}x + b_{\lfloor\lambda\rfloor}) \geq \lim_{\lambda\rightarrow\infty} S_{TI, k, \lambda}(a_{\lfloor\lambda\rfloor}x + b_{\lfloor\lambda\rfloor}) = \lim_{\lambda\rightarrow\infty} S_{TI, k, \lambda}(a_{\lambda}x + b_{\lambda})
    $$
    $$
    \lim_{\lambda\rightarrow\infty} S_{TI, k, \lceil\lambda\rceil}(a_{\lceil\lambda\rceil}x + b_{\lceil\lambda\rceil}) \leq \lim_{\lambda\rightarrow\infty} S_{TI, k, \lambda}(a_{\lceil\lambda\rceil}x + b_{\lceil\lambda\rceil}) = \lim_{\lambda\rightarrow\infty} S_{TI, k, \lambda}(a_{\lambda}x + b_{\lambda})
    $$
    which implies 
    $$
    \lim_{\lambda\rightarrow\infty} S_{TI, k, \lfloor\lambda\rfloor}(a_{\lfloor\lambda\rfloor}x + b_{\lfloor\lambda\rfloor}) = \lim_{\lambda\rightarrow\infty} S_{TI, k, \lceil\lambda\rceil}(a_{\lceil\lambda\rceil}x + b_{\lceil\lambda\rceil}) = \lim_{\lambda\rightarrow\infty} S_{TI, k, \lambda}(a_{\lambda}x + b_{\lambda})
    $$
    On to moments. Let $T_{k, \lambda}$ denote the $k$th passage time of the TII process with immigration rate $\lambda u(t)$ at time $t$. When $m$ is odd,
    $$
        \E[\widetilde{T_{k, \lfloor\lambda\rfloor}}^m] \geq \E\left[\left(\frac{T_{k, \lambda} - b_{\lfloor\lambda\rfloor}}{a_{\lfloor\lambda\rfloor}}\right)^m\right], \quad\quad \E[\widetilde{T_{k, \lceil\lambda\rceil}}^m] \leq \E\left[\left(\frac{T_{k, \lambda} - b_{\lceil\lambda\rceil}}{a_{\lceil\lambda\rceil}}\right)^m\right].
    $$
    Taking limits leads to the desired result. When $m$ is even,
    $$
    \E\left[\left(\frac{T_{k, \lambda} - b_{\lfloor\lambda\rfloor}}{a_{\lfloor\lambda\rfloor}}\right)^m\right] \leq \E\left[\max\left(\widetilde{T_{k, \lfloor\lambda\rfloor}}^m, \left(\frac{T_{k, \lceil\lambda\rceil} - b_{\lfloor\lambda\rfloor}}{a_{\lfloor\lambda\rfloor}}\right)^m\right)\right].
    $$
    Noting that for any two numbers $c$ and $d$, the max of $c$ and $d$ is $(c+d + |c-d|)/2$ and that $\E\left[\left(\frac{T_{k, \lceil\lambda\rceil} - b_{\lfloor\lambda\rfloor}}{a_{\lfloor\lambda\rfloor}}\right)^m\right]$ has the same limit as $\E[\widetilde{T_{k, \lfloor\lambda\rfloor}}^m]$, then taking limits implies
    $$
    \lim_{\lambda\rightarrow \infty}\E[\widetilde{T_{k, \lambda}}^m] = \lim_{\lambda\rightarrow \infty}\E\left[\left(\frac{T_{k, \lambda} - b_{\lfloor\lambda\rfloor}}{a_{\lfloor\lambda\rfloor}}\right)^m\right] \leq \lim_{\lambda\rightarrow \infty}\E[\widetilde{T_{k, \lfloor\lambda\rfloor}}^m].
    $$
    For the lower bound, we apply Fatou's lemma (see Theorem 3.4 in \cite{billingsley2013}). Since the lower bound is the same as the upper bound, we are done.    
\end{proof}

\subsection{Proving Theorem \ref{thm:TIlimitDist}}
The plan is to compare the immigration process with the large $N$ process. To do so, we use a coupling argument outlined in Figure \ref{fig:coupling}. 

\begin{figure}[ht!]
    \centering
    \includegraphics[width=0.7\linewidth]{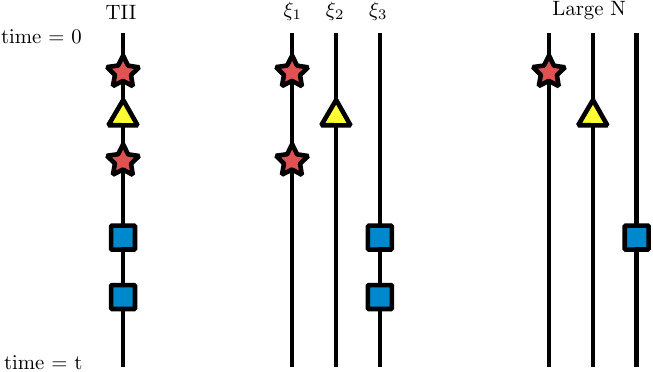}
    \caption{A diagram for the proof of Theorem \ref{thm:TIlimitDist} when $\lambda = 3$. Each vertical line denotes a timeline for a process and the markers represent times a searcher found a target. By Poisson thinning, one can couple the hitting times of TII with rate $\lambda u(t)$ with three TIIs with rate $u(t)$ denoted by $\xi_1, \xi_2,$ and $\xi_3$. If we only record the FPT for each $\xi_i$, we get three iid FPTs, which can be modeled with a large $N$ process with $N=3$ and an adjust survival probability for each searcher. The fastest searcher is going to be the same by construction for TII and large $N$, so $T_1 =  \mathcal{T}_{1}$. For $k\geq 2$ it is not necessarily true that $T_k =  \mathcal{T}_{k}$; one can observe in the diagram that $T_3 \neq  \mathcal{T}_{3}$. However, the probability they are different approaches $0$ as $\lambda \rightarrow \infty$.}
    \label{fig:coupling}
\end{figure}

By Lemmas \ref{lem:betterablambdas} and \ref{lem:naturalnumbers}, we only need to consider integer values of $\lambda$. Let $\xi_i$ be iid TII processes with immigration rate $u(t)$ and each searcher has survival probability $S(t)$. By Poisson superposition, the union of the passage times of $\{\xi_i\}_{i=1}^\lambda$ gives the hitting times of the TII process with immigration rate $\lambda u(t)$. Mathematically, this means that, if we let $\xi_i^{(j)}$ denote the $j$th FPT of the process $\xi_i$ and let $\kmin$ denote the $k$th minimum of a set,
$$
T_{k} = \kmin\bigg(\bigcup_{i=1}^\lambda \bigcup_{j=1}^\infty \{\xi_i^{(j)}\}\bigg).
$$
Now to couple $\{\xi_i\}_{i=1}^\lambda$ to a large $N$ searcher process. For each $\xi_i$, remove all passage times other than the first passage time, $\xi_i^{(1)}$. By properties of time inhomogeneous PPPs, $\xi_i^{(1)}$ has survival probability
$$
P(\xi_i^{(1)}>t) = \exp\left[-\int_0^t u(t-s)(1-S(s)) ds\right] = S_*(t).
$$
This implies that the $k$th smallest $\xi_i^{(1)}$ is the $k$th smallest of $\lambda$ many initial iid searchers, each with survival probability $S_*(t)$, and therefore
$$
 \mathcal{T}_{k} = \kmin\bigg(\bigcup_{i=1}^\lambda \{\xi_i^{(1)}\}\bigg).
$$
Now that we have coupled the TII process with a large $N$ search process through $\{\xi_i\}_{i=1}^\lambda$, we examine the relation between $T_k$ and $\mathcal{T}_{k}$ in two separate cases: $k=1$ and $k\geq 2$. In the case $k=1$, then $\mathcal{T}_{1} = T_{1}$, and therefore $\widetilde{T_{1}}$ converges in distribution to whatever $\widetilde{\mathcal{T}_{1}}$ converges to. Noting that 
\begin{align*}
        1-S_*(t) &= 1-\exp\left[-\int_0^t u(t-s) (1-S(s)) ds \right] \\
        &= 1-\exp\left[-I(t) \right] \\
        &\sim I(t)
\end{align*}
 and we are given $I(t) \sim A_0t^{p_0}$ or $I(t) \sim A_0t^{p_0}e^{-C_0/t}$, applying Theorem \ref{thm:largeNkthFPT} finishes the proof for $k=1$.

In the case of $k\geq 2$, $\mathcal{T}_{k} = T_k$ when the first $k$ arrivals come from different $\xi$s, and $\mathcal{T}_{k} \geq T_k$ otherwise. Since each arrival has equal probability of coming from any of the $\lambda$ immigration processes, 
$$
P(\mathcal{T}_{k} = T_k) = \prod_{j=0}^{k-1}\frac{\lambda-j}{\lambda} \rightarrow 1 \quad\text{as}\quad\lambda \rightarrow \infty.
$$
This implies $\widetilde{T_{k}} - \widetilde{\mathcal{T}_{k}} \rightarrow 0$ in probability, and therefore $\widetilde{T_{k}}$ converges in distribution to whatever $\widetilde{\mathcal{T}_{k}}$ converges to. Using the small time behavior of $1-S_*(t)$ with Theorem \ref{thm:largeNkthFPT} gives the desired result.

\subsection{Showing equation \eqref{eqn:shortTimeIt}}
\label{app:TInumerics}
We first note two facts we will use repeatedly throughout the proof. The first, which follows from a change of variables, $u=C/s$, indicates
$$
\int_0^t As^pe^{-C/s}ds = AC^{p+1}\Gamma(-p-1, C/t).
$$
The second is that the upper incomplete Gamma function has asymptotic behavior
$$
\Gamma(r,z)\sim z^{r-1}e^{-z}\quad\text{as }z\to\infty.
$$
Now back to the problem. By definition
$$
I(t) \sim \int_0^t A_\alpha(t-s)^n As^pe^{-C/t}ds.
$$
Using integration by parts with
$$
u=\alpha_n(t-s)^n, \quad du=-n\alpha_n(t-s)^{n-1},\quad  dv = As^pe^{-C/t}ds, \quad v = AC^{p+1}\Gamma(-p-1, C/t)
$$
it follows that
\begin{align*}
    I(t) &\sim \left[(t-s)^n\alpha_nAC^{p+1}\Gamma(-p-1, C/s)\right]_{s=0}^{s=t} + n \int_0^t \alpha_nAC^{p+1}\Gamma(-p-1, C/s)(t-s)^{n-1}ds \\
    &=n \int_0^t \alpha_nAC^{p+1}\Gamma(-p-1, C/s)(t-s)^{n-1}ds.
\end{align*}
Using the asymptotic behavior of $\Gamma(\cdot, \cdot)$, 
\begin{align*}
    I(t) &\sim n\int_0^t \alpha_nAC^{p+1}\left(\frac{C}{s}\right)^{-p-2}e^{-C/s}(t-s)^{n-1}ds \\
    &=\frac{n}{C}\int_0^t \alpha_nAs^{p+2}e^{-C/s}(t-s)^{n-1} ds.
\end{align*}
This yields a recursive relationship that lets us decrease the degree of $(t-s)$. Repeating this recursion until the degree is $0$, we then find
$$
I(t) \sim \frac{n!}{C^n}\int_0^t \alpha_nAs^{p+2n}e^{-C/s} ds.
$$
Lastly, integrating and using the asymptotic behavior of $\Gamma(\cdot, \cdot)$ yields the desired conclusion
$$
I(t) \sim \frac{A\alpha_nn!}{C^{n+1}}t^{p+2n+2}e^{-C/t}.
$$

\subsection{Proving Theorem \ref{thm:TIlimitMom}}
By Lemmas \ref{lem:betterablambdas} and \ref{lem:naturalnumbers}, we only need to consider integer $\lambda$s. As we showed in the proof of Theorem \ref{thm:TIlimitDist}, $\mathcal{T}_{1} = T_1$ so applying Theorem \ref{thm:largeNFPTMom} with the small time behavior of $1-S_*(t)$ immediately proves the case $k=1$.

We now turn our attention to $k\geq 2$. The outline of the remainder of the proof is that convergence in distribution and uniform integrability imply convergence of mean (see, for example, Theorem~3.5 in \cite{billingsley2013}). As such, it suffices to show that for any positive integer $m$, the set of random variables $\{\widetilde{T_{k}}^m\}_{\lambda>\lambda^*}$ is uniformly integrable.

Recall that the definition of uniform integrability is
$$
\lim_{K\rightarrow \infty} \sup_{\lambda>\lambda^*} \E[|\widetilde{T_{k}}|^m \mathbbm{1}_{\{|\widetilde{T_{k}}|^m > K^m\}}] = 0.
$$
It then suffices to show that
$$
\lim_{K\rightarrow \infty} \sup_{\lambda>\lambda^*} \E[\widetilde{T_{k}}^m \mathbbm{1}_{\{\widetilde{T_{k}} > K\}}] = 0, \qquad \lim_{K\rightarrow \infty} \sup_{\lambda>\lambda^*} \E[\left(-\widetilde{T_{k}}\right)^m \mathbbm{1}_{\{\widetilde{T_{k}} < -K\}}] = 0,
$$
or equivalently
\begin{equation}
    \lim_{K\rightarrow \infty} \sup_{\lambda>\lambda^*} \E\left[\left(\frac{T_{k}-b_{\lambda}}{a_\lambda}\right)^m \mathbbm{1}_{T_{k} > b_\lambda+K a_\lambda\}}\right] = 0,
    \label{eqn:goal1}
\end{equation}
\begin{equation}
    \lim_{K\rightarrow \infty} \sup_{\lambda>\lambda^*} \E\left[\left(\frac{b_{\lambda}-T_{k}}{a_\lambda}\right)^m \mathbbm{1}_{T_{k} < b_\lambda-K a_\lambda\}}\right] = 0.
    \label{eqn:goal2}
\end{equation}

We start with \eqref{eqn:goal2}, the simpler equation to prove. Since $T_{k} \geq T_{1} = \mathcal{T}_{1}$, it follows that
$$
\lim_{K\rightarrow \infty} \sup_{\lambda>\lambda^*} \E\left[\left(\frac{b_{\lambda}-T_{k}}{a_\lambda}\right)^m \mathbbm{1}_{T_{k} < b_\lambda-K a_\lambda\}}\right] \leq \lim_{K\rightarrow \infty} \sup_{\lambda>\lambda^*} \E\left[\left(\frac{b_{\lambda}-\mathcal{T}_{1}}{a_\lambda}\right)^m \mathbbm{1}_{\mathcal{T}_{1} < b_\lambda-K a_\lambda\}}\right].
$$
Since Corollary \ref{coro:largeNUI} from large $N$ results tells us $\{(\mathcal{T}_{k} - b_{\lambda})/a_\lambda\}_{\lambda}$ is uniformly integrable, the righthand side is 0, which is what was desired.

Next, we show \eqref{eqn:goal1}. By the coupling argument given in the proof of Theorem \ref{thm:TIlimitMom}, we find $T_{k} \leq \mathcal{T}_{k}$. Substituting into the lefthand side of \eqref{eqn:goal1}, we find
$$
\lim_{K\rightarrow \infty} \sup_{\lambda>\lambda^*} \E\left[\left(\frac{T_{k}-b_{\lambda}}{a_\lambda}\right)^m \mathbbm{1}_{T_{k} > b_\lambda+K a_\lambda\}}\right] \leq \lim_{K\rightarrow \infty} \sup_{\lambda>\lambda^*} \E\left[\left(\frac{\mathcal{T}_{k}-b_{\lambda}}{a_\lambda}\right)^m \mathbbm{1}_{\mathcal{T}_{k} > b_\lambda+K a_\lambda\}}\right].
$$
Once again, since Corollary \ref{coro:largeNUI} from large $N$ results tells us $\{(\mathcal{T}_{k} - b_{\lambda})/a_\lambda\}_{\lambda}$ is uniformly integrable, the righthand side is 0, which is what was desired.

\section{Proofs of theorems for YI}
\subsection{Prove \eqref{eqn:YIindividual}}
Using the interarrival time formula \eqref{eqn:YuleInterarrival} given in Section \ref{sec:YIimmigrationrates}, the density for $\sigma_{1:n}$ being the immigration times in time interval $(0, t)$ is
\begin{align*}
    f(\sigma_{1:n})&=e^{-\lambda(n+1)(t-\sigma_n)}\prod_{k=1}^n (k\lambda)e^{-(k\lambda)(\sigma_k-\sigma_{k-1})} \\
    &= \lambda^n n! \exp\left(\lambda\sum_{i=1}^n \sigma_i\right)e^{-\lambda(n+1)t}
\end{align*}
where we define $\sigma_0 = 0$.

Further recall from Section \ref{sec:YIimmigrationrates} that the probability of exactly $n$ births (or population size $n+1$) at time $t$ is
$$
\left(1-e^{-\lambda t}\right)^n e^{-\lambda t}.
$$
Conditioning on the number of births, we get the new density
$$
\frac{f(\sigma_{1:n})}{\left(1-e^{-\lambda t}\right)^n e^{-\lambda t}} = n! \prod_{k=1}^n \frac{\lambda e^{\lambda \sigma_i}}{e^{\lambda t}-1}.
$$
The $n\!$ can be seen as a correction factor for rearranging $\sigma_{1:n}$ in other random orders. In other words, when conditioned on the number of births, we can generate $\sigma_{1:n}$ by picking $\sigma_i'$ iid from a truncated exponential random variable with rate $\lambda$, which has distribution,
$$
P(\sigma_i' \leq a) = \frac{e^{\lambda a}-1}{e^{\lambda t}-1},
$$
then reordering $\sigma_{1:n}'$ in ascending order to get $\sigma_{1:n}$.

\subsection{Prove \eqref{eqn:binomGeo}}
One can prove this with directly, but we give a probabilistic proof to avoid some algebra. Noting that one can simulate the number of trials $\geo(p_1)$ by
\begin{enumerate}
    \item Set count = 0
    \item Loop and do exactly one of the following
    \begin{enumerate}
        \item With probability $p_1$, return count.
        \item With probability $1-p_1$, count = count + 1.
    \end{enumerate}
\end{enumerate}
one can simulate $\binom\left(\geo(p_1), p_2\right)$ by adding an additional coin flip with probability $p_2$ to decide if count increases by $1$. 
\begin{enumerate}
    \item Set count = 0
    \item Loop and do exactly one of the following
    \begin{enumerate}
        \item With probability $p_1$, return count.
        \item With probability $(1-p_1)p_2$, set count = count + 1
        \item With remaining probability do nothing.
    \end{enumerate}
\end{enumerate}
If we then condition on either returning count or increasing count by 1 each loop, the algorithm becomes
\begin{enumerate}
    \item Set count = 0
    \item Loop and do exactly one of the following
    \begin{enumerate}
        \item With probability $p_1/(p_1 + (1-p_1)p_2)$, return count.
        \item With probability $(1-p_1)p_2/(p_1 + (1-p_1)p_2)$, set count = count + 1
    \end{enumerate}
\end{enumerate}
This is just the algorithm for simulating $\geo(p_1/(p_1 + (1-p_1)p_2))$, concluding the proof.

\subsection{Proving Theorem \ref{thm:YIlimitDist}}
We first show a lemma that will be of use in proving convergence in distribution and convergence of moments.
\begin{lemma}
    $$\lim_{\lambda \rightarrow \infty} \lambda e^{\lambda  b_\lambda}\int_0^{a_\lambda x + b_\lambda} (1-S(s))e^{-\lambda s} ds = 1$$
    \label{lem:YIpointwise}
\end{lemma}
\begin{proof}
    We consider the two cases $1-S(t)\sim At^p$ and $1-S(t)\sim At^pe^{-C/t}$ separately. When $1-S(t)\sim At^p$, for sufficiently large $\lambda$, $a_\lambda x + b_\lambda$ is close to $0$, which implies
    $$
    \lambda e^{\lambda  b_\lambda}\int_0^{a_\lambda x + b_\lambda} (1-\varepsilon)As^pe^{-\lambda s} ds \leq \lambda e^{\lambda  b_\lambda}\int_0^{a_\lambda x + b_\lambda} (1-S(s))e^{-\lambda s} ds \leq \lambda e^{\lambda  b_\lambda}\int_0^{a_\lambda x + b_\lambda} (1-\varepsilon)As^pe^{-\lambda s} ds
    $$
    As such, it suffices to show
    $$\lim_{\lambda \rightarrow \infty} \lambda e^{\lambda  b_\lambda}\int_0^{a_\lambda x + b_\lambda} As^pe^{-\lambda s} ds = 1$$
    Recalling the gamma function is $\Gamma(s) = \int_0^\infty t^{s-1}e^{-t} dt$, we note
    \begin{align*}
        \lim_{\lambda \rightarrow \infty} \lambda e^{\lambda  b_\lambda}\int_0^{a_\lambda x + b_\lambda} As^pe^{-\lambda s} ds &= \lim_{\lambda \rightarrow \infty} \lambda e^{\lambda  b_\lambda}\int_0^{\lambda (a_\lambda x + b_\lambda)} A\left(\frac{u}{\lambda}\right)^pe^{-u} du\\
        &=\lim_{\lambda \rightarrow \infty} A\lambda^{-p}e^{\lambda b_\lambda} \int_0^{\lambda (a_\lambda x + b_\lambda)} u^pe^{-u} du \\
        &=\lim_{\lambda \rightarrow \infty} A\lambda^{-p}e^{\lambda b_\lambda} \Gamma(p+1) \\
        &=1.
    \end{align*}

    Next, we examine the case $1-S(t)\sim At^pe^{-C/t}$. Similar to the previous case, it suffices to show
    $$
    \lim_{\lambda \rightarrow \infty} \lambda e^{\lambda  b_\lambda}\int_0^{a_\lambda x + b_\lambda} As^pe^{-C/s}e^{-\lambda s} ds = 1.
    $$
    Taking the substitution $s = u\sqrt{C/\lambda}$, we simplify to
    $$
    \sqrt{C\lambda}e^{\lambda b_\lambda}\int_{0}^{(a_\lambda x+b_\lambda)\sqrt{\lambda/C}}  A\left(\sqrt{\frac{C}{\lambda}}u\right)^pe^{-\sqrt{C\lambda}\left(u+\frac{1}{u}\right)} du.
    $$
    Plugging in the definition of $a_\lambda$ and $b_\lambda$ and simplifying, we find
    $$
    \sqrt{\frac{\Lambda}{\pi}}\int_0^{2+\frac{x+(p-\frac{1}{2})\ln(\Lambda)-\ln(AC^p\sqrt{\pi})}{\Lambda}} u^pe^{-\Lambda \left(\sqrt{u} - \frac{1}{\sqrt{u}}\right)^2} du
    $$
    where $\Lambda := \sqrt{C\lambda}$. 

    Note that $(\sqrt{u} - \frac{1}{\sqrt{u}})^2$ achieves its unique global minimum $0$ at $u=1$, which for sufficiently large $\Lambda$, is in the boundaries of integration. As such, it follows from Laplace's Method (see 6.4.19c in \cite{bender2013}) that the integral has leading order behavior $\sqrt{\pi/\Lambda}$, and therefore
    $$
    \sqrt{\frac{\Lambda}{\pi}}\int_0^{2+\frac{x+(p-\frac{1}{2})\ln(\Lambda)-\ln(AC^p\sqrt{\pi})}{\Lambda}} u^pe^{-\Lambda \left(\sqrt{u} - \frac{1}{\sqrt{u}}\right)^2} du \sim 1\quad\text{as $\Lambda \rightarrow \infty$}.
    $$
\end{proof}

Now to prove Theorem~\ref{thm:YIlimitDist}. Note that
\begin{align*}
    \lim_{\lambda \rightarrow \infty}S_{YI, k}(a_\lambda x + b_\lambda) &= \lim_{\lambda \rightarrow \infty} S(a_\lambda x + b_\lambda)[1-(1-p(a_\lambda x + b_\lambda))^k]\\
    &\phantom{=\lim_{\lambda \rightarrow \infty}} + (1-S(a_\lambda x + b_\lambda))[1-(1-p(a_\lambda x + b_\lambda))^{k-1}] \\
    &= \lim_{\lambda \rightarrow \infty} [1-(1-p(a_\lambda x + b_\lambda))^k]\\
    &= \lim_{\lambda \rightarrow \infty} \left[1-\left(1-\frac{1}{1 + \lambda e^{\lambda (a_\lambda x + b_\lambda)}\int_0^{a_\lambda x + b_\lambda} (1-S(s))e^{-\lambda s} ds}\right)^k\right]\\
    &= \lim_{\lambda \rightarrow \infty} \left[1-\left(1-\frac{1}{1 + e^x\lambda e^{\lambda b_\lambda}\int_0^{a_\lambda x + b_\lambda} (1-S(s))e^{-\lambda s} ds}\right)^k\right]\\
    &= 1-\left(1-\frac{1}{1 + e^x}\right)^k
\end{align*}
where the first equality follows from Proposition \ref{prop:SYIexact}, the second from $a_\lambda$ and $b_\lambda$ converging to $0$, and therefore $S(a_\lambda x + b_\lambda) \rightarrow 1$, the third again from Proposition \ref{prop:SYIexact}, the fourth from $a_\lambda = 1/\lambda$, and the last equality from Lemma \ref{lem:YIpointwise}.

\subsection{Proving theorem \ref{thm:YImoments}}
Since we have convergence in distribution, it suffices to show uniform integrability of $\{\widetilde{T_{k}}^m\}_{\lambda > \lambda^*}$ (see, for example, Theorem~3.5 in \cite{billingsley2013}). $\lambda^*$ may be chosen to be arbitrarily large as necessary.

Recall that the definition of uniform integrability is
$$
\lim_{K\rightarrow \infty} \sup_{\lambda>\lambda^*} \E[|\widetilde{T_k}|^m \mathbbm{1}_{\{|\widetilde{T_k}|^m > K^m\}}] = 0.
$$
First, we bound the expected value for some $K>0$. Noting that $|\widetilde{T_k}|^m \mathbbm{1}_{\{|\widetilde{T_k}|^m > K^m\}}$ is a nonnegative random variable, its $m$th moment is given by
\begin{align*}
    \E[|\widetilde{T_1}|^m \mathbbm{1}_{\{|\widetilde{T_k}|^m > K^m\}}] &= \int_0^{\infty} mx^{m-1}P(|\widetilde{T_k}| \mathbbm{1}_{\{|\widetilde{T_k}|^m > K^m\}} > x) dx \\   
    &= \int_K^{\infty} mx^{m-1}P(|\widetilde{T_k}| > x) dx \\ 
    &= \int_K^{\infty} mx^{m-1}(P(\widetilde{T_k} > x) + 1 -P(\widetilde{T_k} > -x)) dx.
\end{align*}
Noting that
\begin{align*}
    P(\widetilde{T_k} > x)
    &= P(T_k > a_{\lambda} x + b_{\lambda})
    = S_{YI, k}(a_{\lambda} x + b_{\lambda}),\\
    P(\widetilde{T_k} > -x)
    &= P(T_k > -a_{\lambda} x + b_{\lambda})
    = S_{YI, k}(-a_{\lambda} x + b_{\lambda}),
\end{align*}
we can decompose $\E[|\widetilde{T_k}|^m \mathbbm{1}_{\{|\widetilde{T_k}|^m > K^m\}}]$ into the sum of the following two terms,
\begin{align}
    &\int_K^{\infty} mx^{m-1} S_{YI, k}(a_{\lambda} x + b_{\lambda}) dx 
    \label{eqn:YIlimitingproofint1},\\
    &\int_K^{\infty} mx^{m-1} (1-S_{YI, k}(-a_{\lambda} x + b_{\lambda})) dx. 
    \label{eqn:YIlimitingproofint2}
\end{align}
We first tackle \eqref{eqn:YIlimitingproofint1}. 
\begin{lemma}
    For sufficiently large $\lambda^*$
    $$
    \lim_{K\rightarrow \infty} \sup_{\lambda > \lambda^*} \int_K^{\infty} mx^{m-1} S_{YI, k}(a_{\lambda} x + b_{\lambda}) dx =0.
    $$
    \label{lem:YIlimitingproofint1}
\end{lemma}
\begin{proof}
    Since $1-(1-p(t))^k \geq 1-(1-p(t))^{k-1}$ for $k\geq 1$ and $S_{YI, k}(t)$ is a biased average between the two values, it follows that $S_{YI, k}(t) \leq 1-(1-p(t))^k$ and therefore
    \begin{align*}
        \int_K^{\infty} mx^{m-1} S_{YI, k}(a_{\lambda} x + b_{\lambda}) dx &\leq \int_K^{\infty} mx^{m-1} [1-(1-p(a_{\lambda} x + b_{\lambda}))^k] dx \\
        &= \int_K^{\infty} mx^{m-1} \left[ p(a_{\lambda} x + b_{\lambda})\sum_{i = 0}^{k-1}(1-p(a_{\lambda} x + b_{\lambda}))^{i} \right] dx \\
        &\leq k \int_K^{\infty} mx^{m-1} p(a_{\lambda} x + b_{\lambda}) dx \\
        &= k \int_K^{\infty} mx^{m-1} \frac{1}{1 + \lambda e^{\lambda (a_{\lambda} x + b_{\lambda})}\int_0^{a_{\lambda} x + b_{\lambda}} (1-S(s))e^{-\lambda s} ds} dx \\
        &= k \int_K^{\infty} mx^{m-1} \frac{1}{1 + e^x \lambda e^{\lambda b_{\lambda}}\int_0^{a_{\lambda} x + b_{\lambda}} (1-S(s))e^{-\lambda s} ds} dx \\
        &\leq k \int_K^{\infty} mx^{m-1} \frac{1}{1 + e^x \left( \lambda e^{\lambda b_{\lambda}}\int_0^{ b_{\lambda}} (1-S(s))e^{-\lambda s} ds \right)} dx.
    \end{align*}
    Applying Lemma \ref{lem:YIpointwise} with $x=0$, there exists a $\lambda^*$ such that if $\lambda>\lambda^*$, then $\lambda e^{\lambda b_{\lambda}}\int_0^{ b_{\lambda}} (1-S(s))e^{-\lambda s} ds >1-\varepsilon$ and therefore
    \begin{align*}
        \int_K^{\infty} mx^{m-1} S_{YI, k}(a_{\lambda} x + b_{\lambda}) dx &\leq k \int_K^{\infty} mx^{m-1} \frac{1}{1 + e^x (1-\varepsilon)} dx \\
        &\leq k \int_K^{\infty} mx^{m-1} e^{-x}\frac{1}{ 1-\varepsilon} dx.
    \end{align*}
    Since this upper bound is finite for $K=0$ and does not depend on $\lambda$, it follows that the upper bound approaches $0$ as $K\rightarrow \infty$.
\end{proof}

To finish the proof, we tackle \eqref{eqn:YIlimitingproofint2}.
\begin{lemma}
    For sufficiently large $\lambda^*$
    $$
    \lim_{K\rightarrow \infty} \sup_{\lambda > \lambda^*} \int_K^{\infty} mx^{m-1} (1-S_{YI, k}(b_{\lambda}-a_{\lambda} x)) dx =0
    $$
    \label{lem:YIlimitingproofint2}
\end{lemma}
\begin{proof}
    We first note that since $1-S_{YI, k}(b_{\lambda}-a_{\lambda} x) = 0$ when $x\leq b_\lambda/a_\lambda$, so it suffices to show
    $$
    \lim_{K\rightarrow \infty} \sup_{\lambda > \lambda^*} \int_K^{b_\lambda/a_\lambda} mx^{m-1} (1-S_{YI, k}(b_{\lambda}-a_{\lambda} x)) dx =0.
    $$
    Since the first searcher arrives faster than the $k$th searcher for $k\geq 2$, it follows that $S_{YI, k}(t) \geq S_{YI, 1}(t)$ and therefore
    \begin{align*}
        1-S_{YI, k}(b_{\lambda}-a_{\lambda} x) dx &\leq 1-S_{YI, 1}(b_{\lambda}-a_{\lambda} x) \\
        &=1-S(b_{\lambda}-a_{\lambda} x)p(b_{\lambda}-a_{\lambda} x) \\
        &=1-p(b_{\lambda}-a_{\lambda}x)+p(b_{\lambda}-a_{\lambda} x)(1-S(b_{\lambda}-a_{\lambda} x)) \\
        &\leq [1-p(b_{\lambda}-a_{\lambda}x)]+[1-S(b_{\lambda}-a_{\lambda} x)]
    \end{align*}
    As such, it suffices to show
    $$
    \lim_{K\rightarrow \infty} \sup_{\lambda > \lambda^*} \int_K^{b_\lambda/a_\lambda} mx^{m-1} [1-p(b_{\lambda}-a_{\lambda}x)] dx =0
    $$
    $$
    \lim_{K\rightarrow \infty} \sup_{\lambda > \lambda^*} \int_K^{b_\lambda/a_\lambda} mx^{m-1} [1-S(b_{\lambda}-a_{\lambda} x)] dx =0
    $$
    The second equation holds by previous work (see equation A.10 in \cite{tung2025first}). As such, we only need to examine the first equation. Note that
    \begin{align*}
        1-p(b_{\lambda}-a_{\lambda}x) dx &= 1-\frac{1}{1 + \lambda e^{\lambda (b_{\lambda}-a_{\lambda}x)}\int_0^{b_{\lambda}-a_{\lambda}x} (1-S(s))e^{-\lambda s} ds}\\
        &= \frac{1}{1 + \frac{1}{\lambda e^{\lambda (b_{\lambda}-a_{\lambda}x)}\int_0^{b_{\lambda}-a_{\lambda}x} (1-S(s))e^{-\lambda s} ds}}\\
        &= \frac{1}{1 + e^x\frac{1}{\lambda e^{\lambda b_{\lambda}}\int_0^{b_{\lambda}-a_{\lambda}x} (1-S(s))e^{-\lambda s} ds}}\\
        &\leq \frac{1}{1 + e^x\frac{1}{\lambda e^{\lambda b_{\lambda}}\int_0^{b_{\lambda}} (1-S(s))e^{-\lambda s} ds}}\\
        &\leq \frac{1}{1 + e^x\frac{1}{1+\varepsilon}}\\
        &\leq (1+\varepsilon)e^{-x},
    \end{align*}
    which implies
    $$
    \lim_{K\rightarrow \infty} \sup_{\lambda > \lambda^*} \int_K^{b_\lambda/a_\lambda} mx^{m-1} [1-p(b_{\lambda}-a_{\lambda}x)] dx \leq \lim_{K\rightarrow \infty} \sup_{\lambda > \lambda^*} \int_K^{b_\lambda/a_\lambda} mx^{m-1} (1+\varepsilon)e^{-x} dx =0
    $$
\end{proof}
By lemmas \ref{lem:YIlimitingproofint1} and \ref{lem:YIlimitingproofint2}, we have uniform integrability and therefore convergence of moments to the moments of the limiting distribution.

\subsection*{Data availability statement}

The data that supports the findings of this study are available from the corresponding author upon reasonable request.


\bibliography{library.bib}
\bibliographystyle{unsrt}

\end{document}